\documentclass[11pt]{amsart}
\usepackage{amsmath,amsthm,amscd,amsfonts,amssymb,epic,eepic,bbm,graphicx,ytableau,tikz,tikz-cd}
\usetikzlibrary{angles, quotes}
\usepackage{empheq}
\usepackage{booktabs}  
\usepackage{eqparbox}  

\allowdisplaybreaks

\newtheorem{thm}{Theorem}[section]
\newtheorem{cor}[thm]{Corollary}
\newtheorem{conj}[thm]{Conjecture}
\newtheorem{lem}[thm]{Lemma}
\newtheorem{defn}[thm]{Definition}

\newtheorem{prop}[thm]{Proposition}

\newtheorem{rem}[thm]{Remark}
\newtheorem{example}[thm]{Example}
\theoremstyle{remark}

\newcommand{\la}{\lambda}
\newcommand{\<}{\langle}
\newcommand{\ra}{\rangle}

\newcommand{\rib}{\mathfrak{r}}

\begin{document}
 
\title[chromatic basis and $p$-positivity of skew Schur $Q$-functions]{An analogue of chromatic bases and\\ $p$-positivity of skew Schur $Q$-functions} 

\author{Soojin Cho}
\address{Department of Mathematics, Ajou University, Suwon 16499, Republic of Korea}
\email{chosj@ajou.ac.kr}
\thanks{The first author and the second author were supported by Basic Science Research Program through the National Research Foundation of Korea
(NRF) funded by the Ministry of Education (NRF-2015R1D1A1A01057476).}

\author{JiSun Huh}
\address{Applied Algebra and Optimization Research Center, Sungkyunkwan University, Suwon 16419, Republic of Korea}
\email{hyunyjia@g.skku.edu}

\author{Sun-Young Nam}
\address{Departement of Mathematics, Sogang University, Seoul 04107, Republic of Korea}
\email{synam.math@gmail.com}
\thanks{The third author was supported by Basic Science Research Program through the National Research Foundation of Korea(NRF) funded by the Ministry of Education(NRF-R1D1A1B03030945 and NRF-2019R1I1A1A01062658).}

\subjclass[2010]{}

\thanks{
}

\begin{abstract} 
We investigate chromatic symmetric functions in the relation to the algebra $\Gamma$ of symmetric functions generated by Schur $Q$-functions. We construct natural bases of $\Gamma$ in terms of chromatic symmetric functions. We also consider the $p$-positivity of skew Schur $Q$-functions and find a class of $p$-positive ribbon Schur $Q$-functions, making a conjecture that they are \emph{all}. We include many concrete computational results that support our conjecture. 
\end{abstract}

\maketitle

\section{Introduction}\label{sec:intro}

Stanley introduced \emph{chromatic symmetric functions} as a symmetric function generalization of chromatic polynomials of finite simple graphs in 1995 \cite{S1}. They provided many new directions of research in the relation with many different areas in mathematics, including graph theory, symmetric function theory, representation theory and algebraic geometry. 

It is natural to look at chromatic symmetric functions in the relation of other well known symmetric functions, and Stanley-Stembridge conjecture on $e$-positivity of certain chromatic symmetric functions is one of such problems. In \cite{CW1, CW2}, chromatic symmetric function bases of the symmetric function space $\Lambda$ were constructed and classical symmetric functions that can be realized as chromatic symmetric functions were classified. 

The purpose of current paper is to investigate chromatic symmetric functions in the relation with the algebra $\Gamma$ of symmetric functions generated by Schur $Q$-functions, answering the same questions raised in  \cite{CW1, CW2} for the algebra of symmetric functions. However, only null graphs (with isolated vertices) make symmetric functions in $\Gamma$, and we define \emph{near} chromatic symmetric functions and consider the generators and bases of $\Gamma$ instead, which make natural analogues of chromatic symmetric function bases of $\Lambda$. 
An interesting property of a chromatic symmetric function is that its image under the well known involution $\omega$ on $\Lambda$ is $p$-positive, and the classification of classical symmetric functions that are also chromatic symmetric functions was done by classifying $p$-positive skew Schur functions (see Proposition~\ref{prop:p-positive Schur}).
This suggests another interesting problem to ask which skew Schur $Q$-functions are $p$-positive, and we work on this also. We show that only ribbon Schur $Q$-functions can be $p$-positive and found a class of $p$-positive ribbons. We come to make a nice conjecture that the ones we found are \emph{all} with many computational evidences. The shapes of $p$-positive ribbons are obtained by repeating (near) concatenation of the transpose of a shape from a basic block, and by taking their transpose or antipodal rotation. See Conjecture~\ref{necessary}.

The rest of the current paper is organized as follows. Section~\ref{sec:diagrams} is to introduce  background for our work and Section~\ref{sec:basis} is devoted to the construction of chromatic bases of $\Gamma$. In Section~\ref{sec:p-positivity}, $p$-positivity of skew Schur $Q$-functions are considered, and in the final section we make final remarks.

\section{Preliminaries}\label{sec:diagrams}

In this section we give  basic definitions, setup notations and introduce known results that we will need for the development of our arguments.

A positive integer sequence $\alpha=(\alpha_1,\alpha_2,\dots,\alpha_{\ell})$ is called a \emph{composition $\alpha$ of $n$}, denoted $\alpha \vDash n$, if $\sum_{i}\alpha_i=n$. Given a composition $\alpha=(\alpha_1,\alpha_2,\dots,\alpha_{\ell})$ of $n$, we call the $\alpha_i$ the \emph{parts} of $\alpha$, $\ell=\ell(\alpha)$ the \emph{length} of $\alpha$, and $n=|\alpha|$ the \emph{size} of $\alpha$. We introduce a partial order on compositions which will be used. For two compositions $\alpha=(\alpha_1,\alpha_2,\dots,\alpha_{\ell(\alpha)})$ and $\beta=(\beta_1,\beta_2,\dots,\beta_{\ell(\beta)})$ such that $\alpha, \beta \vDash n$, we say that $\alpha$ is a \emph{coarsening} of $\beta$ (or $\beta$ is a \emph{refinement} of $\alpha$), denoted $\alpha \succcurlyeq \beta$, if consecutive parts of $\beta$ can be added together to yield the parts of $\alpha$. For example, $(5,2,3)\succcurlyeq(2,2,1,2,1,2)$.
  
A \emph{partition $\la=(\la_1, \la_2, \dots, \la_\ell)$ of $n$}, denoted $\la \vdash n$, is a composition of $n$ satisfying $\la_1\geq \la_2\geq \cdots \geq \la_{\ell}$, and is said to be \emph{strict} if $\la_1>\la_2>\dots>\la_{\ell}$. We also use $\la=(1^{m_1}2^{m_2}\cdots)$ as an alternative notation for partitions, where $m_i=m_i(\la)$ is the number of appearance of the part $i$ in the partition $\la$. Given two partitions $\la$ and $\mu$, we say that $\mu$ is \emph{contained} in $\la$, denoted $\mu \subseteq \la$, if $\mu_i\leq \la_i$ for all $i=1,2,\dots,\ell(\mu)$. We let $OP(n)$ be the set of all partitions of $n$ into odd parts
 and $SP(n)$ be the set of all strict partitions of $n$. The partition obtained by reordering the parts of a composition $\alpha$ is denoted by $\la(\alpha)$. 

\subsection{Diagrams}\label{subsec:diagrams}
The \emph{(Young) diagram} is a graphical interpretation of a partition $\la$, which is the array of left justified boxes containing $\la_i$ boxes in the $i$th row from the top. We abuse notation by denoting the diagram of a partition $\la$ by $\la$ also. For partitions $\la$ and $\mu$ such that $\mu\subseteq \la$, the \emph{skew diagram} $\la/\mu$ is obtained from the diagram $\la$ by removing the boxes of the diagram $\mu$ from the top left box.

 For a strict partition $\la$, the \emph{shifted diagram} of $\la$, denoted $\widetilde{\la}$, is defined by
\[
\widetilde{\la}=\{(i,j) \,|\, 1\leq i \leq \ell(\la),~ i\leq j \leq i+\la_i -1 \} ,
\]
that is obtained from the diagram $\la$ by shifting the $i$th row $(i-1)$ boxes to the right, for each $i>1$.
Similarly, for strict partitions $\la$ and $\mu$ such that $\mu \subseteq \la$, the {\it shifted skew diagram}, denoted $\widetilde{\la / \mu}$, is obtained from $\la/\mu$ by shifting the $i$th row from the top $(i-1)$ boxes to the right for $i>1$. 

We are dealing with diagrams either $\la/\mu$ or $\widetilde{\la / \mu}$, that are determined by two partitions $\mu\subseteq\la$, but we specify the partitions involved or the type of the diagram only when they are needed. We usually use $D$ to denote a diagram and understand it as a set of boxes(coordinates) in the plane. 
For a (shifted) skew diagram, each edgewise connected part is called a \emph{component}. A diagram with one component is called \emph{connected}. 
A skew diagram $D$ is said to be a {\it ribbon}, or a {\it border strip}, if for each box $(i,j)$ in $D$ the box $(i-1,j-1)$ is \emph{not} contained in $D$. {If a connected ribbon $D$ of size $n$ has $\alpha_i(D)$ boxes in the $i$th row for each $i=1, 2, \dots, \ell$, we correspond the ribbon to the composition $\alpha(D)=(\alpha_1(D),\alpha_2(D),\dots, \alpha_\ell(D))$ of $n$. We abuse the notation by denoting the ribbon by $\alpha(D)$ and we call $\ell(\alpha(D))$ the \emph{length of the ribbon $D$.}}

\begin{example}
If $\la=(4,2,2)$ and $\mu=(1,1)$, then the skew diagram $\la/\mu$ is the ribbon $\alpha=(3,1,2)$. If $\la=(4,3,2)$ and $\mu=(3,2)$, then the shifted skew diagram $\widetilde{\la/\mu}$ is the ribbon $\beta=(1,1,2)$.

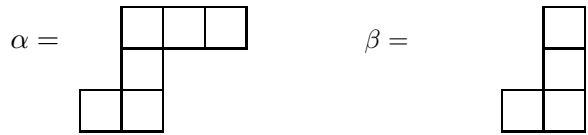
\begin{center}
$\alpha=$
\small{
\begin{ytableau}
\none&~&~&~\\
\none&~\\
~&~
\end{ytableau}
\qquad \qquad
$\beta=$
\begin{ytableau}
\none&\none&\none&~\\
\none&\none&\none&~\\
\none&\none&~&~
\end{ytableau}}
\end{center}
\end{example}

\subsection{Operations on skew diagrams}\label{composition of transposition}

We begin with recalling two classical operations on skew diagrams. For a skew diagram $D$, the \emph{transpose} $D^t$ is obtained from $D$ by reflecting in the main diagonal, and the \emph{antipodal rotation} $D^{\circ}$ is obtained by rotating $D$ 180 degrees in the plane.  

Let $D_1$ and $D_2$ be skew diagrams. The \emph{disjoint union} of $D_1$ and $D_2$, denoted $D_1 \oplus D_2$, is obtained by placing $D_1$ strictly north and east of $D_2$ such that they have no common row or column. Given $D_1\oplus D_2$, the \emph{concatenation} $D_1 \cdot D_2$ (resp. \emph{near concatenation} $D_1 \odot D_2$) is obtained by moving all boxes of $D_1$ exactly one cell west (resp. south). For example, for ribbons $\alpha=(\alpha_1,\alpha_2,\dots,\alpha_k)$ and $\beta=(\beta_1,\beta_2,\dots,\beta_{\ell})$, we have $\alpha \cdot \beta=(\alpha_1,\dots,\alpha_k,\beta_1,\dots,\beta_{\ell})$ and $\alpha \odot \beta=(\alpha_1,\dots,\alpha_k+\beta_1,\dots,\beta_{\ell})$.

We note that both $\cdot$ and $\odot$ are associative and they associate with each other, and a ribbon $\alpha$ with $k$ boxes can be uniquely written as 
\[
\alpha=\square \bigstar_1 \square \bigstar_2 \cdots \bigstar_{k-1} \square ,  
\]
where $\square$ denotes the diagram with one box and each $\bigstar_i$ is either $\cdot$ or $\odot$. For example, the ribbon $\alpha=(3,1,2)$ can be written as $\square \odot \square \odot \square \cdot \square \cdot \square \odot \square$.\\
The main operation we deal with is the composition of transposition, which was introduced in \cite{BW}. For a composition $\alpha$ and a skew diagram $D$, the \emph{composition of transposition { of $D$ by $\alpha$}} is defined as  
\[
\alpha \bullet D=\begin{cases}
    D \bigstar_1 D^t \bigstar_2 D \bigstar_3 D^t \cdots \bigstar_{k-1} D & \text{if~} |\alpha| \text{~is odd}, \\
    D \bigstar_1 D^t \bigstar_2 D \bigstar_3 D^t \cdots \bigstar_{k-1} D^t & \text{if~} |\alpha| \text{~is even},
    \end{cases}
\]
where $\alpha=\square \bigstar_1 \square \bigstar_2 \cdots \bigstar_{k-1} \square$.

\begin{example}
If { $\alpha^{(1)}=(1,1)$, $\alpha^{(2)}=(2)$, and $D=(1,3)$, the composition of transposition $\alpha^{(1)} \bullet D$ of $D$ by $\alpha^{(1)}$ is the ribbon $(1,3,1,1,2)$, and $\alpha^{(2)} \bullet \alpha^{(1)} \bullet D$ }is the ribbon $(1,3,1,1,3,4,1,2)$ as shown in the following figure, in which gray boxes represent transpose:\\

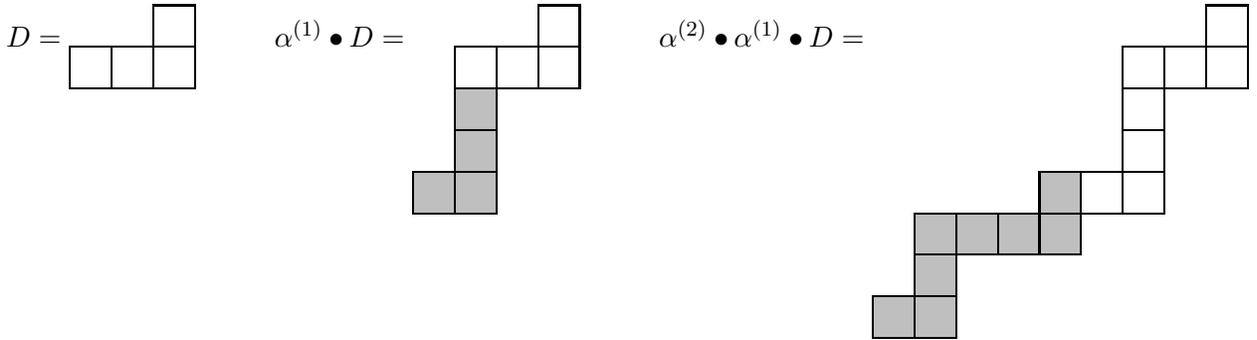
\begin{center}
$D=
\small{
\begin{ytableau}
\none&\none&~\\
~&~&~\\
\none\\
\none\\
\none\\
\none\\
\none\\
\none
\end{ytableau}
}$\qquad\,\,
$\alpha^{(1)} \bullet D=
\small{
\begin{ytableau}
\none&\none&\none&~\\
\none&~&~&~\\
\none&*(gray!50)\\
\none&*(gray!50)\\
*(gray!50)&*(gray!50)
\end{ytableau}
}$\qquad\,\,
$\alpha^{(2)}\bullet\alpha^{(1)} \bullet D=
\small{
\begin{ytableau}
\none&\none&\none&\none&\none&\none&\none&\none&~\\
\none&\none&\none&\none&\none&\none&~&~&~\\
\none&\none&\none&\none&\none&\none&~\\
\none&\none&\none&\none&\none&\none&~\\
\none&\none&\none&\none&*(gray!50)&~&~\\
\none&*(gray!50)&*(gray!50)&*(gray!50)&*(gray!50)\\
\none&*(gray!50)\\
*(gray!50)&*(gray!50)
\end{ytableau}
}$
\end{center}
\end{example}

\subsection{Algebra $\Lambda$ of symmetric functions}

For a positive integer $n$ and a partition $\la=(\la_1, \dots, \la_\ell)$ of $n$, the \emph{monomial symmetric function} corresponding to $\la$ is defined as
$m_\la=\sum_{(i_1,\dots, i_\ell)} x_{i_1}^{\la_1}\cdots x_{i_\ell}^{\la_\ell}\,,$
where the sum is over the $\ell$ tuples of distinct positive integers. Monomial symmetric functions are invariant under the action of permuting variables and they form a basis of the \emph{space $\Lambda^n$ of $n$th degree symmetric functions}; $\Lambda^n=\mathrm{span}_\mathbb Q \{ m_\la\,| \, \la \vdash n\}.$
The \emph{algebra $\Lambda$ of symmetric functions} is the graded subalgebra of $\mathbb Q[[x_1, x_2, \dots]]$ defined as
$ \Lambda=\oplus_{n\geq 0} \,\, \Lambda^n, \text{ where we set }  \Lambda_0=\mathbb Q\,.$

For a positive integer $r$, the \emph{$r$th elementary symmetric function} is $e_r=\sum_{i_1<\cdots < i_r} x_{i_1} \cdots x_{i_r}$,  the \emph{$r$th complete homogeneous symmetric function} is $h_r=\sum_{i_1\leq\cdots \leq i_r} x_{i_1} \cdots x_{i_r}$, and the  \emph{$r$th power sum symmetric function} is $p_r=\sum_i x_i^r$. Then each of $\{e_\la\,| \, \la \vdash n\}$, $\{h_\la\,| \, \la \vdash n\}$ and $\{p_\la\,| \, \la \vdash n\}$ forms a nice basis of $\Lambda^n$, where $b_\la=b_{\la_1}\cdots b_{\la_\ell}$ for $\la=(\la_1, \dots, \la_\ell)$ and $b\in \{ e, h, p\}$. Another important basis of $\Lambda^n$ is the set $\{ s_\la\,|\, \la\vdash n\}$ of \emph{Schur functions}, where $s_\la=\mathrm{det}( h_{\la_i-i+j})_{1\leq i,j \leq \ell(\la)}$.

There is a nice combinatorial model for Schur functions. For a given partition $\la$, a \emph{semistandard tableau  $T$ of shape $\la$} is a filling of the diagram of $\la$ with positive integers so that each row is weakly increasing and each column is strictly increasing, and the \emph{content of $T$} is $c(T)=(c_1, c_2, \dots)$ where $c_i$ is the number of $i$ in $T$.  Then $s_\la=\sum_T x^{c(T)}$, where the sum is over all semistandard tableaux of shape $\la$ and $x^{(c_1, c_2, \dots)}=\prod_i x_i^{c_i}$. We can also define a \emph{skew Schur function $s_{\la/\mu}$} for $\mu\subseteq \la$ as the generating function of semistandard tableaux of shape $\la/\mu$.

We say that a given symmetric function is \emph{$b$-positive} if the function can be written as a positive linear combination of the elements of the basis $\{b_{\la}\}$ of $\Lambda$. We remark that in the algebra of symmetric functions $\Lambda$, there is a well known \emph{involution} $\omega$ that is defined by $\omega(p_r)=(-1)^{r-1}p_{r}$ or equivalently by  $\omega(s_\la)=s_{\la^t}$, and there is a \emph{scalar product} defined by $\< s_{\la},s_{\mu} \ra=\delta _{\la \mu}$ or equivalently by
$\< p_{\la},~p_{\mu} \ra=\delta _{\la \mu} z_{\la},$ where $z_{\la}=\prod_i i^{m_i}(m_i)!$ for  $\la=(1^{m_1}2^{m_2}\cdots)$. 

\subsection{Subalgebra $\Gamma$ of symmetric functions}

Odd power sum symmetric functions make an algebraically independent set over $\mathbb Q$ and generate an important subalgebra of $\Lambda$. $$\Gamma=\mathbb Q [ p_1,\, p_3,\, p_5, \dots ] =\oplus_{n\geq 0} \Gamma^n\,,\text{ where } \Gamma^n=\Gamma\cap \Lambda^n\,.$$
The subalgebra $\Gamma$ of $\Lambda$ has nice generators and bases elements that are different from the power sum symmetric functions.  We now define the skew Schur $Q$-functions as generating functions of marked shifted tableaux as Schur functions are the generating functions of semistandard tableaux.

Let $\mathbf{P}'$ denote the set of ordered alphabet $\{1'<1<2'<2<\cdots \}$. Here, the letters $1', 2', \dots$ are said to be \emph{marked}. For a letter $a \in \mathbf{P}'$, we denote the unmarked version of this letter by $|a|$. Let $\la$ and $\mu$ be strict partitions such that $\mu\subseteq \la$. A \emph{marked shifted tableau} $T$ of shape $\widetilde{\la / \mu}$ is a filling of the boxes of the shifted skew diagram $\widetilde{\la / \mu}$ with letters from $\mathbf{P}'$ such that 
\begin{enumerate}
\item[(M1)] each row  and column are weakly increasing,
\item[(M2)] each column has at most one $k$ for each $k\geq 1$, and
\item[(M3)] each row has at most one $k'$ for each $k\geq 1$.
\end{enumerate}
The {\it content} of $T$ is the sequence $c(T)=(c_1,c_2, \dots)$, where $c_i$ is the number of all letters $a$ such that $|a|=i$ in $T$, for each $i\geq 1$.

\begin{defn}
For strict partitions $\la$ and $\mu$ with $\mu \subseteq \la$, the \emph{skew Schur $Q$-function} is defined as 
\[
Q_{\la / \mu} = \sum_{T} x^{c(T)} ,
\]
where the sum is over all marked shifted tableaux of shape $\widetilde{\la / \mu}$. 
We denote $Q_{\la / \emptyset}$ by $Q_{\la}$ and $Q_{(n)}$ by $q_n$,  
{where we let $Q_\emptyset=q_0=1$ for convenience.} For a partition $\la=(\la_1, \la_2, \dots, \la_\ell)$, we let $q_\la=\prod_{i=1}^\ell q_{\la_i}$.
\end{defn}

Our main interest is on the algebra $\Gamma$ and we summarize some important results.

\begin{prop} \cite[Chapter III]{Mac} \label{prop:SchurQ}
For a positive integer $n$, 
\begin{equation}\label{eqn:relation}
{\sum_{r=0}^n (-1)^r q_r q_{n-r} =0\,,}
\end{equation}
\begin{equation}\label{eqn:q_to_p}
q_n=\sum_{\la\in OP(n)}z_{\la}^{-1}2^{\ell(\la)}p_{\la}\,\in \Gamma\,.
\end{equation}
\end{prop}

\begin{thm} \cite[Chapter III]{Mac} \label{thm:gamma}
\begin{enumerate}
\item[(a)] $\Gamma=\mathbb{Q}[ q_1,\, q_3,\, q_5, \dots ]$, 
and $\{q_1, q_3, q_5, \dots \}$ is an algebraically independent set over $\mathbb{Q}$. 
\item[(b)] For any strict partitions $\la$ and $\mu$ such that $\mu\subseteq \la$, $Q_{\la/\mu}$ is a symmetric function in $\Gamma$.
\item[(c)] For each $n\geq 1$, $\{Q_\la\,|\, \la\in SP(n)\}$ forms a basis of $\Gamma^n$.
\item[(d)] For each $n\geq 1$, both $\{p_\la\,|\, \la\in OP(n)\}$ and $\{q_\la\,|\, \la\in OP(n)\}$ are bases of $\Gamma^n$.
\end{enumerate}
\end{thm}

For two strict partitions $\la$ and $\mu$ such that $\mu\subseteq \la$, if $\widetilde{\la/\mu}$ is a ribbon then we follow the convention in \cite{BW} to use $\mathfrak r_{\alpha}$ for $Q_{\la/\mu}$ when $\alpha$ is the composition representing the ribbon of shape $\la/\mu$. In this case we also use $\alpha^t$ and  $\alpha^\circ$ for the transpose and the antipodal rotation of  $\widetilde{\la/\mu}$, respectively.
We associate an $\ell \times \ell$ matrix $A(\alpha)$ to a composition(ribbon) $\alpha=(\alpha_1, \dots, \alpha_\ell)$, where the $(i, j)$ entry of $A(\alpha)$ is given by
\begin{equation}\label{eqn:matrix} A(\alpha)_{i\,j}=\begin{cases}  { q_{\,\alpha_i+\alpha_{i+1}+\cdots + \alpha_j}} & \text{ if\,\, } i\leq j ,\\
                                     {q_0} & \text{ if\,\, } j-i=1 ,\\  0 & \text{ otherwise}\,.
               \end{cases}                                                                  
\end{equation}

\begin{prop}\label{prop:ribbon} \cite[Proposition 3.3]{BW}
Let $\alpha$ and $\beta$ be ribbons. Then,
\begin{equation}\label{eqn:circle}
\mathfrak r_{\alpha}=\mathfrak r_{\alpha^t}=\mathfrak r_{\alpha^\circ}
\end{equation}
\begin{equation}\label{eqn:det}
\mathfrak r_{\alpha}=\det(A(\alpha))                                                                            
\end{equation}
\begin{equation}\label{eqn:ribbon}
\mathfrak r_{\alpha}=(-1)^{\ell(\alpha)}\sum_{\gamma \succcurlyeq \alpha}(-1)^{\ell(\gamma)}q_{\la(\gamma)}
\end{equation}
\begin{equation}\label{eqn:ribbonop}
\mathfrak r_{\alpha} \mathfrak r_{\beta}=\mathfrak r_{\alpha \cdot \beta}+\mathfrak r_{\alpha \odot \beta}
\end{equation}
\begin{equation}\label{eqn:ribbonsq}
\mathfrak r_{\alpha}^{\, 2} = 2\mathfrak r_{\alpha \cdot \alpha^t} 
=  2\mathfrak r_{\alpha \odot \alpha^t}\, .
\end{equation}
\end{prop}

\section{A basis of the algebra $\Gamma$ involving graphs}\label{sec:basis}

In this section we construct new bases for the algebra $\Gamma$, whose elements are generated by the chromatic symmetric functions of the star graphs and the triangles.
These are analogues of the chromatic symmetric function bases for the algebra $\Lambda$ introduced in \cite{CW1}.

We begin by defining the chromatic symmetric functions.\\
Let $G=(V,E)$ be a simple graph with a vertex set $V=\{v_1, v_2, \dots, v_n\}$ and an edge set $E$. A function $\kappa : V \rightarrow \{1,2,\dots\}$ is called a \emph{proper coloring} of $G$ if $\{v_i,v_j\} \in E$ implies that $\kappa(v_i) \neq \kappa(v_j)$. The \emph{chromatic symmetric function of $G$} is defined as
\[
X_G =\sum_{\kappa} x_{\kappa (v_1)}\dots x_{\kappa (v_n)},
\]
where the sum is over all proper colorings of $G$. 
It is an immediate consequence of the definition of $X_G$ that if a graph $G$ is a disjoint union of subgraphs $G_1, G_2,\dots, G_{\ell}$, then we have $X_G=\prod_{i=1}^{\ell}X_{G_i}$. Moreover, $X_{G}$ is a symmetric function and one can expand $X_G$ in terms of the basic bases of symmetric functions. 
As we consider the algebra $\Gamma=\mathbb Q [p_1, p_3, p_5, \dots]$ in this paper, we are especially interested in the expansion of chromatic symmetric functions into power sum symmetric functions $p_{\la}$ among other basic bases. Stanley found a nice explicit $p$-expansion formula of the chromatic symmetric functions. 

\begin{thm}\label{thm:chro} \cite[Theorem 2.5, Corollary 2.7]{S1} 
For a simple graph $G=(V,E)$ with $n$ vertices, let $\la (S)$ be the partition of $n$ whose parts are equal to the numbers of vertices in the connected components of the spanning subgraph of $G$ with an edge set $S$. Then, the chromatic symmetric function of $G$ can be written as
\[
X_G=\sum_{S\subseteq E} (-1)^{|S|}p_{\la (S)},
\]
and the symmetric function $\omega(X_G)$ is $p$-positive.
\end{thm}

\begin{example}\label{exam:chro}\cite[Theorem 8]{CW1}
We give the chromatic symmetric functions $X_G$ of two exemplary graphs $G$.
\begin{enumerate}
    \item[(a)] For the \emph{triangle} $C_3$, the cycle with three vertices, we have
         \[X_{C_3}=p_1^3-3p_2p_1+2p_3\,.\]
    \item[(b)] For the \emph{star graph} $S_n$, the tree with one internal vertex and $n-1$ leaves, we have
         \[
         X_{S_n}=\sum_{r=0}^{n-1}(-1)^r \binom{n-1}{r}p_{r+1}p_{1}^{n-r-1}\,.
         \]
\end{enumerate}
\end{example}

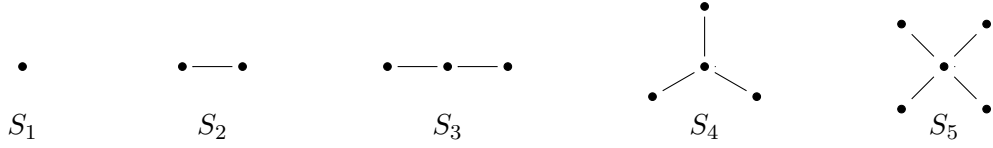
\begin{figure}
\centering
\begin{tikzpicture}[scale=0.8]
\node at (0,1) {};
\filldraw[black] (0,0) circle  (1.8pt);
\node at (0,-1) {$S_1$};
\end{tikzpicture}
\qquad\qquad
\begin{tikzpicture}[scale=0.8]
\node at (0,1) {};
\node  (0) at (0,0) {};
\node  (1) at (1,0) {};

\foreach \i in {0,1}
{
\draw (0) -- (\i);
\filldraw[black] (\i) circle  (1.8pt);
}
\node at (0.5,-1) {$S_2$};
\end{tikzpicture}
\qquad\qquad
\begin{tikzpicture}[scale=0.8]
\node at (0,1) {};
\node  (0) at (0,0) {};
\node  (1) at (1,0) {};
\node  (2) at (-1,0) {};

\foreach \i in {0,1,2}
{
\draw (0) -- (\i);
\filldraw[black] (\i) circle  (1.8pt);
}
\node at (0,-1) {$S_3$};
\end{tikzpicture}
\qquad\qquad
\begin{tikzpicture}[scale=0.8]
\node at (0,1) {};
\node  (0) at (0,0) {};
\node  (1) at ({cos(330)},{sin(330)}) {};
\node  (2) at ({cos(90)},{sin(90)}) {};
\node  (3) at ({cos(210)},{sin(210)}) {};

\foreach \i in {0,1,2,3}
{
\draw (0) -- (\i);
\filldraw[black] (\i) circle  (1.8pt);
}
\node at (0,-1) {$S_4$};
\end{tikzpicture}
\qquad \qquad
\begin{tikzpicture}[scale=0.8]
\node at (0,1) {};
\node  (0) at (0,0) {};
\node  (1) at ({cos(45)},{sin(45)}) {};
\node  (2) at ({cos(135)},{sin(135)}) {};
\node  (3) at ({cos(225)},{sin(225)}) {};
\node  (4) at ({cos(315)},{sin(315)}) {};

\foreach \i in {0,1,2,3,4}
{
\draw (0) -- (\i);
\filldraw[black] (\i) circle  (1.8pt);
}
\node at (0,-1) {$S_5$};
\end{tikzpicture}
    \caption{Star graphs.}  \label{fig:star}
\end{figure}

In \cite{CW1}, the authors obtained numerous new bases for the algebra of symmetric functions whose generators are chromatic symmetric functions.

\begin{thm}\label{thm:chrobase} \cite[Theorem 5]{CW1} Let $G_{\la}=G_{\la_1} \cup \cdots \cup G_{\la_{\ell}}$ denote a graph with connected components $G_{\la_1},\dots, G_{\la_{\ell}}$ such that $G_k$ has $k$ vertices for each $k$. Then the set
\[
\{X_{G_{\la}}\,|\,\la \vdash n\}
\]
forms a $\mathbb{Q}$-basis of $\Lambda^n$.
\end{thm}

It is natural to ask that for which graphs $G$, a set $\{X_{G}\}$ forms a basis of $\Gamma^n$. However, it follows from Theorem~\ref{thm:chro} that $X_G$ dose not belongs to $\Gamma$ if $G$ has an edge. 

\begin{prop} \label{prop:noedge} 
For a finite simple graph $G$, it has no edge if and only if $X_G \in \Gamma$. 
\end{prop}
\begin{proof}
Let $G$ be a simple graph with $n$ vertices. If $G$ has no edge, then $X_G=p_1^n \in \Gamma$.

We now let $G$ be a graph having $m$ edges such that $X_G \in \Gamma$. By Theorem~\ref{thm:chro}, we have 
\[
X_G=p_1^n -m\, p_2 p_1^{n-2} +\sum_{|S|\geq 2}(-1)^{|S|}p_{\la(S)}\,.
\]
Since $p_n$, $n\geq 0$,  are algebraically independent and $\{p_1,p_3,p_5,\dots\}$ generates $\Gamma$, for $X_G$ to be in $\Gamma$ the second term $m\,p_2p_1^{n-2}$ must be $0$, that is $m=0$. Hence, $G$ has no edge.
\end{proof}

A graph with no edge is called a \emph{null graph}.
Due to Proposition~\ref{prop:noedge}, the only chromatic symmetric functions contained in $\Gamma$ are $X_G=p_1^n$ for the null graphs $G$ with $n$ vertices and there can not be a chromatic symmetric function basis for $\Gamma$. A symmetric function $f$ in $\Gamma$ satisfies $\omega(f)=f$ and we consider the \emph{symmetrization} of the chromatic symmetric functions instead.  
For a simple graph $G$, we define a \emph{near chromatic symmetric function of $G$} to be 
\[
Y_G=(X_G+\omega(X_G))/2 \,.
\]

Then it is clear from the definition that $\omega(Y_G)=Y_G$. We also note that if $X_G$ belongs to $\Gamma$, then $Y_{G}=X_{G}$ and it follows that $\omega(Y_G)=Y_G$ is $p$-positive. From Example~\ref{exam:chro}, we obtain two near chromatic symmetric functions which belong to the algebra $\Gamma$.

\begin{prop} \label{prop:nearchro} For a graph $G$, let $Y_{G}$ be its near chromatic symmetric function. 
\begin{enumerate}
    \item[(a)] For the triangle $C_3$, we have 
    \[Y_{C_3}=p_1^3+2p_3\,.\]
    \item[(b)] For the star graph $S_n$, we have 
         \[
         Y_{S_n}=\begin{cases}
\binom{n-1}{0}p_{1}^n+\binom{n-1}{2}p_3p_{1}^{n-3}+\cdots + \binom{n-1}{n-1}p_{n} & \text{if~} n \text{~is odd},\\[2ex] 
\binom{n-1}{0}p_1^{n}+\binom{n-1}{2}p_3p_{1}^{n-3}+\cdots + \binom{n-1}{n-2}p_{n-1}p_1 & \text{if~} n \text{~is even}\,.\\[1.2ex]
\end{cases}
         \]
Moreover, \{$Y_{S_n}\,|\, n \text{: odd }\}$ is algebraically independent over $\mathbb{Q}$.   
\end{enumerate}
\end{prop}

As the near chromatic symmetric function $Y_G$ is obtained from the $p$-expansion of $X_G$ by getting rid of all terms $p_{\la}$ for partitions $\la$ with an odd number of even parts, $Y_G$ can belong to $\Gamma$ for some graphs $G$ although $X_G$ is not an element of $\Gamma$. It turns out that there is no connected graph $G$ such that $Y_{G}\in\Gamma$ except $C_3$ and $S_n$ for $n\geq1$.\\

We say that two edges $\{v_1, v_2\}$ and $\{v_3, v_4\}$ are \emph{disjoint} if $v_1, v_2, v_3$, and $v_4$ are all distinct.

\begin{lem} \label{lem:disjoint}
Let $G$ be a simple graph such that $Y_G\in\Gamma$. Then $G$ has no pair of disjoint edges.
Furthermore, if $G$ is connected, then $G$ is either $C_3$ or $S_n$ for $n\geq 1$.
\end{lem}

\begin{proof}
We suppose that there are $r$ different ways to select two disjoint edges of $G$. By Theorem~\ref{thm:chro}, the coefficient of $p_2^2p_1^{n-4}$ in the $p$-expansion of  $X_G$ is equal to $r$. Hence, from the definition of $Y_G$, the coefficient of $p_2^2p_1^{n-4}$ in the $p$-expansion of $Y_{G}$ is also $r$, and if $Y_G \in \Gamma$ then $r$ must be zero so that $G$ has no pair of disjoint edges. It is easy to see that a connected graph with no pair of disjoint edges is either $C_3$ or $S_n$ for $n\geq 1$.
\end{proof}

From Lemma~\ref{lem:disjoint}, we can see that $Y_G$ need not be the same as $\prod_{i}Y_{G_i}$, where $G_i$ are connected components of $G$. We now classify all graphs $G$ such that $Y_G\in\Gamma$. 

\begin{prop} \label{prop:classify}
Let $G$ be a simple graph. Then, $Y_G\in\Gamma$ if and only if $G$ is a disjoint union of a graph $H$ and a  null graph, where $H$ is one of the graphs $C_3$ and $S_n$ for $n\geq 1$. 
\end{prop}

\begin{proof}
We first note that if $G$ is a disjoint union of $H$ and the null graph with $d$ vertices, then    
\[
Y_{G}=p_1^{d}\,Y_{H}\in\Gamma \,.
\]

Now, we suppose that $G$ is a simple graph such that $Y_G \in \Gamma$. From Lemma~\ref{lem:disjoint}, $G$ has no pair of disjoint edges. Hence, all edges are in the same connected component of $G$, say $H$, so that $G$ is a disjoint union of $H$ and a null graph. Moreover, $H$ is either $C_3$ or $S_n$ for $n\geq 1$ since $H$ is a connected graph with no pair of disjoint edges.
\end{proof}

We have come to the conclusion that there are only two generator sets of near chromatic symmetric functions for the algebra $\Gamma$ which are algebraically independent:

\begin{thm} 
Two sets $\{Y_{S_1}, Y_{C_3}, Y_{S_5}, Y_{S_7}, \dots\}$ and $\{Y_{S_1}, Y_{S_3}, Y_{S_5}, Y_{S_7}, \dots\}$ of near chromatic symmetric functions are algebraically independent generator sets for the algebra $\Gamma$. Moreover, they are the only algebraically independent generator sets for $\Gamma$ consisting of near chromatic symmetric functions.
\end{thm} 
\begin{proof} Since $\Gamma$ is generated by odd power sum symmetric functions and $p_k$, for odd $k$ are algebraically independent, it is immediate from Proposition~\ref{prop:nearchro} that two given sets are algebraically independent and
\[
\Gamma=\mathbb{Q}[Y_{S_1}, Y_{C_3}, Y_{S_5}, Y_{S_7}, \dots]
     =\mathbb{Q}[Y_{S_1}, Y_{S_3}, Y_{S_5}, Y_{S_7}, \dots]\,.
\]
Proposition~\ref{prop:classify} proves that there is no other algebraically independent generator set of near chromatic symmetric functions.
\end{proof}

We close this section by giving new bases for the space $\Gamma^n$ of $n$th degree symmetric functions in $\Gamma$. For a given set $\mathbf{B}=\{G_1, G_3, G_5, \dots\}$ of simple graphs where $G_k$ has $k$ vertices for $k=1, 3, 5, \dots$, we define a set $\mathcal{Y}(\mathbf{B})\subseteq \Gamma$ of symmetric functions in $\Gamma$ as $\mathcal{Y}(\mathbf{B})=\{\prod_{i}Y_{G_{\la_i}}\,|\,\la\in OP(n)\}$.

\begin{thm} \label{thm:graph}
Let $S_n$ be the star graph with $n$ vertices for $n\geq1$ and $C_3$ be the triangle. Let $\mathbf{B}_1=\{S_1, C_3, S_5, S_7, \dots \}$ and $\mathbf{B}_2=\{S_1, S_3, S_5, S_7, \dots \}$. Then $\mathcal{Y}(\mathbf{B}_1)$ and $\mathcal{Y}(\mathbf{B}_2)$ are bases of $\Gamma^n$.
\end{thm}
\begin{proof} 
We know from Proposition~\ref{prop:nearchro} that $\mathcal{Y}(\mathbf{B}_1)$ and $\mathcal{Y}(\mathbf{B}_2)$ are linearly independent. Since both sets have $|OP(n)|$ elements and $\mathrm{dim} \Gamma=|OP(n)|$ by Theorem~\ref{thm:gamma}, the proof is completed.
\end{proof}


\section{On $p$-positivity of skew Schur $Q$-functions}\label{sec:p-positivity}

There are many positivity questions in the theory of symmetric functions, mainly on Schur positivity, $e$-positivity or $h$-positivity, which are raised usually in the relation to the representation theory, but not so many questions on $p$-positivity. However, in the work related to the \emph{chromatic symmetric functions} $p$-positivity problems naturally appear. 
This is basically because $X_G$ has an interesting property that $\omega(X_G)$ is $p$-positive. 
For instance, in the work to classify classical symmetric functions that can be recognized as chromatic symmetric functions in \cite{CW2}, 
knowing the $p$-positivity of skew Schur functions played an important role.

\begin{prop}\cite[Proposition 2.6]{CW2} \label{prop:p-positive Schur}
A skew Schur function $s_D$ is $p$-positive if and only if $D$ is a horizontal strip.
\end{prop}

In this section we work on the Schur $Q$-analogue of Proposition~\ref{prop:p-positive Schur}; that is, on the classification of \emph{$p$-positive skew Schur $Q$-functions}.  
Throughout the section, any shifted skew diagram $\widetilde{\la/\mu}$ is regarded to be connected unless specifically noted. We recall from Theorem~\ref{thm:gamma} (d), that a skew Schur $Q$-function can be written as a linear sum of \emph{odd} power sum symmetric functions, and also recall the scalar product on $\Lambda$ satisfying $\< s_{\la},s_{\mu} \ra=\delta _{\la \mu}$ and 
$\< p_{\la},p_{\mu} \ra=\delta _{\la \mu} z_{\la}$. To begin with, we show that 
a skew Schur $Q$-function $Q_{\la/\mu}$ is not $p$-positive unless $\widetilde{\la/\mu}$ is a ribbon.

\begin{prop}\label{Q_not_ribbon}
If a shifted skew diagram $\widetilde{\la/\mu}$ contains three boxes $(i,j),~(i,j+1)$, and $(i+1,j+1)$ for some $i$ and $j$, then the skew Schur $Q$-function $Q_{\la/\mu}$ is not $p$-positive.
\end{prop}

The above proposition is immediate from the following lemma.

\begin{lem}\label{lem:coef}
Let $\widetilde{\la/\mu}$ be a shifted skew diagram of size $n$, {that contains} three boxes $(i,j),~(i,j+1)$, and $(i+1,j+1)$ for some $i$ and $j$. If $Q_{\la/\mu}=\sum_{\la \in OP(n)} a_{\la} p_{\la}$,
then $\sum_{\la \in OP(n)}a_{\la}=0$. 
\end{lem}
\begin{proof}
{We first note that there is no marked shifted tableau of shape $\widetilde{\la/\mu}$ with weight $(n)$, since $\widetilde{\la/\mu}$ contains {two diagonally consecutive boxes} $(i,j)$ and $(i+1,j+1)$. This implies that $\< Q_{\la/\mu},s_{(n)}\ra=0$ since $x_1^n$ can appear only in $s_{(n)}$ among Schur functions.
Due to  Murnaghan-Nakayama rule \cite[Corollary 7.17.5]{S2} the Schur function of one row shape is expanded as a sum of power sum symmetric functions as follows;}
$$s_{(n)}=\sum_{\la \vdash n} z_{\la}^{-1}p_{\la} \, .$$
Therefore, we have 
$$\left\< \sum_{\la \in OP(n)}a_{\la}p_{\la},~\sum_{\la \vdash n}z_{\la}^{-1} p_{\la} \right\ra=\sum_{\la \in OP(n)}a_{\la}=0 \, .$$

\end{proof}

\subsection{{A class of $p$-positive ribbon Schur $Q$-functions}}
{In Proposition~\ref{Q_not_ribbon}, we showed that only ribbon Schur $Q$-functions can be $p$-positive. In this subsection we construct an infinite family of $p$-positive ribbon Schur $Q$-functions. }

For positive integers $n$ and $k$ such that $ k \leq n$,
we denote by $\triangle_{n,k}$ the ribbon whose shape is {$\widetilde{\la/\mu}$}, where $\la=(n,n-1,\cdots, n-k+1)$ and $\mu=(n-1,n-2,\cdots,n-k+1)$ if $k>1$, $\mu=\emptyset$ if $k=1$. For example, $\triangle_{n,1}$ is the one-row diagram of size $n$,
and $\triangle_{n,n}$ is the one-column diagram of size $n$ (see Figure~\ref{fig:Basicblock}).

\begin{lem} \label{Q_triangle_to_q}
Let $n$ and $k$ be positive integers such that $ k \leq n$.
Then we have the followings.
\begin{enumerate}
\item[(a)] $\displaystyle{\rib_{\triangle_{n,k}}=\sum_{i=0}^{k-1}(-1)^{k+i-1}q_{n-i}\,q_{i}}$.
\item[(b)] $\rib_{\triangle_{n,n-k+1}}=\rib_{\triangle_{n,k}}$.
\end{enumerate}
\end{lem}

\begin{proof}
Since ${\triangle_{1,n}}^t = \triangle_{n,n}$, by Proposition~\ref{prop:ribbon}\,\eqref{eqn:circle} we have $\rib_{\triangle_{n,1}} = \rib_{\triangle_{n,n}}$. For given $n$ and $k$, we let $\alpha_{n, k}$ be the composition corresponding to the ribbon $\triangle_{n, k}$ and $A(\alpha_{n, k})$ be the $k \times k$ matrix associated to $\alpha_{n, k}$, that is defined in \eqref{eqn:matrix}.

(a) We use an induction on $k$. 
For $k = 1$, by the definition of skew Schur $Q$-functions $\rib_{\triangle_{n,1}}=q_nq_0$.
Suppose that the assertion is true for $k \leq n-2 $.
By Proposition~\ref{prop:ribbon}\,\eqref{eqn:det}, 
we have 
\begin{eqnarray*}
\rib_{\triangle_{n,k+1}}&=& \det(A(\alpha_{n, k+1}))\\
&=& 
(-1)\det(A(\alpha_{n,k})) + q_{n-k}\det(A(\alpha_{k,k})) \\
&=& (-1) \cdot \rib_{\triangle_{n,k}} \, + \, q_{n-k} \cdot \rib_{\triangle_{k,k}}\\
&=&\sum_{i=0}^{k-1}(-1)^{(k+1)+i-1}q_{n-i}\,q_{i} \, + \, q_{n-k}\,q_k \, = \, \sum_{i=0}^{k}(-1)^{(k+1)+i-1}q_{n-i}\,q_{i} \,.
\end{eqnarray*}
Here the fourth equality follows from the induction hypothesis and the observation
$\rib_{\triangle_{k,1}} = \rib_{\triangle_{k,k}}$.

(b)
By using (a) and Proposition~\ref{prop:SchurQ}\,\eqref{eqn:relation}, we have
\begin{eqnarray*}
\rib_{\triangle_{n,n-k+1}}-\rib_{\triangle_{n,k}}
&=& \, \sum_{i=0}^{n-k}(-1)^{n-k+i}q_{n-i}\,q_i \, - \, \sum_{i=0}^{k-1}(-1)^{k+i-1}q_{n-i}\,q_i\\
&=& \,  (-1)^{n-k}\sum_{i=0}^{n}(-1)^{i}q_{n-i}q_i=0 \, .
\end{eqnarray*}
\end{proof}

We state a useful lemma that is an easy consequence of Proposition~\ref{prop:SchurQ}\,\eqref{eqn:q_to_p}.

\begin{lem}\label{lem:simple} 
Let $\la$ be an odd partition and $\mu$ be a partition of a positive integer $n$. If $\la$ is not a refinement of $\mu$, then 
$$\<p_{\la},\,q_{\mu}\ra=0\,.$$ 
\end{lem}

Due to Lemma~\ref{Q_triangle_to_q} and Proposition~\ref{prop:SchurQ}\,\eqref{eqn:q_to_p},
we can write {the $p$-expansion of} $\rib_{\triangle_{n,k}}$ by hand
for sufficiently small or large integers $k$.
The cases where $k = 1$ and $k = 3$ are particularly noteworthy
because the ribbon Schur $Q$-functions corresponding to
${\triangle_{n,1}}$ and  ${\triangle_{n,3}}$ are $p$-positive (see Proposition~ \ref{Q_basic}).
Furthermore,
we anticipate that
all ribbons {whose} corresponding ribbon Schur $Q$-function {is} $p$-positive
are constructible from ${\triangle_{n,1}}$ and  ${\triangle_{n,3}}$
(see Theorem~\ref{sufficient} and Conjecture~\ref{necessary}).

\begin{prop} \label{Q_basic}
Let $n$ be a positive integer.
\begin{enumerate}
    \item[(a)] For $n\geq1$, $\displaystyle{\mathfrak r_{\triangle_{n,1}} = \sum_{\la \in OP(n)} z_{\la}^{-1}2^{\ell(\la)}  \,p_{\la}}\,.$
    \item[(b)] For $n\geq 3$, $\displaystyle{\mathfrak r_{\triangle_{n,3}} = \sum_{\la \in OP(n)} c_\lambda \,p_{\la}\,}$,
where 
\begin{equation*}
c_\la \, = \,
\begin{cases}
z_{\la}^{-1}\,2^{\ell(\la)} & \text{if} \ m_1(\la)=0,\\[1.2ex]
0 & \text{if} \ m_1(\la) = 1 \ \text{or} \ 2 ,\\[1.2ex]
     \binom{m_1(\la)-1}{2}\,z_{\la}^{-1}\,2^{\ell(\la)} & \text{if} \ m_1(\la)\geq 3\,.
\end{cases}
\end{equation*}
\end{enumerate}

Consequently, $\rib_{\triangle_{n,1}}$ and $\rib_{\triangle_{n,3}}$ are $p$-positive.
\end{prop}

\begin{proof}
(a) 
Since $\rib_{\triangle_{n,1}} = q_n$ by Lemma~\ref{Q_triangle_to_q}\,(a), it is immediate from Proposition~\ref{prop:SchurQ}\,\eqref{eqn:q_to_p}.

(b) We recall Lemma~\ref{Q_triangle_to_q}\,(a) and Proposition~\ref{prop:SchurQ}\,\eqref{eqn:q_to_p}; 
\[
\rib_{\triangle_{n,k}}=\sum_{i=0}^{k-1}(-1)^{k+i-1}q_{n-i}q_{i}\qquad \text{and} \qquad q_n=\sum_{\la\in OP(n)}z_{\la}^{-1}2^{\ell(\la)}p_{\la}\,.
\]
Using Lemma~\ref{Q_triangle_to_q}\,(a), for odd partitions $\la$, we have 
\[
z_{\la}c_{\la}=\<p_{\la},\,\rib_{\triangle_{n,3}}\ra = \<p_{\la},\,q_n - q_{n-1}q_1 + q_{n-2}q_2\ra\,
\]
which is determined according to the numbers $m_1(\la)$ and $\ell(\la)$ as follows: 
\begin{itemize}
    \item If $m_1(\la)=0$, then we have
    \[
    z_{\la}c_{\la}=\<p_{\la},\,q_n \ra =2^{\ell(\la)}\,.
    \]
    \item If $m_1(\la)=1$, then we have
    \[
    z_{\la}c_{\la}=\<p_{\la},\,q_n - q_{n-1}q_1 \ra=2^{\ell(\la)}(1-1)=0\,.
    \]
    \item If $m_1(\la)\geq 2$, then we have
    \begin{eqnarray*}
    z_{\la}c_{\la}=\<p_{\la},\,q_n - q_{n-1}q_1 + q_{n-2}q_2\ra         &=&2^{\ell(\la)}\{1-m_1(\la)+m_1(\la)(m_1(\la)-1)/2\}\,\\
    &=&2^{\ell(\la)}(m_1(\la)-1)(m_1(\la)-2)/2\,.
    \end{eqnarray*}
\end{itemize}
In each case, the first equality follows from Lemma~\ref{lem:simple} and the second equality follows from Proposition~\ref{prop:SchurQ} \eqref{eqn:q_to_p}.
\end{proof}

We classify \emph{$p$-positive} ribbons of the form $\triangle_{n,k}$ in the following theorem.

\begin{thm}\label{Q_delta} Let $n$ and $k$ be positive integers such that
$ n \geq k$.
\begin{enumerate}
\item[(a)] For $n\leq 2$, $\rib_{\triangle_{n,k}}$ is $p$-positive. 
\item[(b)] For odd $n\geq 3$, $\rib_{\triangle_{n,k}}$ is $p$-positive if and only if $k\in\{1,\,3,\,n-2,\,n\}$.
\item[(c)] For even $n\geq 4$, $\rib_{\triangle_{n,k}}$ is $p$-positive if and only if $k\in\{1,\,3,\,n/2,\,n/2+1,\,n-2,\,n\}$.
\end{enumerate}
\end{thm}

\begin{proof} By Lemma~\ref{Q_triangle_to_q} (a), $\triangle_{1,1}=q_1=2p_1$ and $\triangle_{2,1}=\triangle_{2,2}=q_2=2 p_1^2$. Hence, (a) is proved.\\
 For $n\geq 3$, due to Lemma~\ref{Q_triangle_to_q}\,(b) and Proposition~\ref{Q_basic}, it suffices to show that the assertion is true for $1<k\leq \lceil n/2 \rceil$ and $k\neq3$. To show that a ribbon Schur $Q$-function $\rib_{\triangle_{n,k}}$ is not $p$-positive, we show that the value $z_{\la}c_\la$ is negative for some specific $\la$ when we let 
$\rib_{\triangle_{n,k}} = \sum_{\lambda \in OP(n)} c_\la \, p_\la$.
 To calculate the value $z_{\la}c_{\la}$, we again use Lemma~\ref{Q_triangle_to_q}\,(a), Lemma~\ref{lem:simple}, and Proposition~\ref{prop:SchurQ}\,\eqref{eqn:q_to_p} as in the proof of Proposition~\ref{Q_basic}\,(b).

(b) Let $n\geq3$ be an odd integer. We first show that for even $k$, the value $z_{(n)}c_{(n)}$ is negative so that $\rib_{\triangle_{n,k}}$ is not $p$-positive;
\[
z_{(n)}c_{(n)}=\< p_{(n)},\,\rib_{\triangle_{n,k}}\ra=\< p_{(n)},\,(-1)^{k-1}q_n\ra
=\< p_{(n)},\,(-1)^{k-1}z_{(n)}^{-1}2^1p_{(n)}\ra=(-1)^{k-1}2 \, =-2 \,.
\] 
We now assume that $k\geq 5$ is odd and show that $z_{\la}c_{\la}$ is negative for $\la=\la((k-2)^{2} \, (n-2k+4)^1)$.

If $k-2>n-2k+4$, then $\la=((k-2)^{2} \, (n-2k+4)^1)$ and we have 
\[
z_{\la}c_{\la}=\<p_{\la},\,\rib_{\triangle_{n,k}} \ra
       =\<p_{\la},\,q_n-q_{2k-4}q_{n-2k+4}-q_{n-k+2}q_{k-2}\ra=2^3(1-1-2)=-16\,.
\]

If $k-2 \leq n-2k+4$, then $\la=((n-2k+4)^1\,(k-2)^2)$ and we have 
\[
z_{\la}c_{\la}=\<p_{\la},\,\rib_{\triangle_{n,k}} \ra
       =\<p_{\la},\,q_n-q_{n-k+2}q_{k-2}\ra=
       \begin{cases}
       2^3(1-3)=-16 & \text{if } k-2=n-2k+4,\\[2ex]
       2^3(1-2)=-8 & \text{ otherwise.}
       \end{cases}
\]

(c) We first show that $\rib_{\triangle_{2k,k}}$ is $p$-positive. Let $\alpha=(1^k)$ be a composition of $k$. Since $\rib_{\alpha}=q_{k}$ and $\triangle_{2k,k}=\alpha \odot \alpha^t$, it follows from Proposition~\ref{prop:ribbon}\,\eqref{eqn:ribbonsq} that $\rib_{\triangle_{2k,k}}=q_k^2/2$ which is $p$-positive.

To complete the proof, we have to show that $\rib_{\triangle_{n,k}}$ is not $p$-positive for $1 < k < n/2$ and $k \neq 3$.\\ 
For even $k$, let $\la=(n/2,\,n/2)$ if $n/2$ is odd and $\la=(n/2+1,\,n/2-1)$ otherwise, then the value
\[
z_{\la}c_{\la}=\<p_{\la},\,\rib_{\triangle_{n,k}} \ra
       =\<p_{\la},\,-q_n\ra
\]
is negative in both cases.\\
For odd $k$, we show that $z_{\la}c_{\la}$ is negative for $\la=(n-k-4,k-2,3,3)$.
Note that $n-k-4 \geq k-2 \geq 3$.

\begin{itemize}
    \item If $k=5$, then $\la=(n-9,3,3,3)$ and the value 
    \[
       z_{\la}c_{\la}=\<p_{\la},\,\rib_{\triangle_{n,k}} \ra
       =\<p_{\la},\,q_n-q_{n-3}q_{3}\ra
    \]
    is equal to $2^4(1-4)=-48$ if $n=12$ and $2^4(1-3)=-32$ otherwise. 
    \item If $k>5$, then we have 
    \[
       z_{\la}c_{\la}=\<p_{\la},\,\rib_{\triangle_{n,k}} \ra
       =\<p_{\la},\,q_n-q_{n-3}q_{3}+q_{n-6}q_{6}-q_{n-k+2}q_{k-2}\ra
    \]
    is equal to $2^4(1-2+1-2)=-32$ if $n-k-4=k-2$ and $2^4(1-2+1-1)=-16$ otherwise.
\end{itemize}

\end{proof}

We classified all $p$-positive ribbons of the form $\triangle_{n,k}$, and Proposition~\ref{prop:ribbon} \eqref{eqn:circle} tells us that their transpose, antipodal rotation, and antipodal rotation of transpose are all $p$-positive.
For each $n \geq 1$, we define $\mathfrak{B}_n$, called the set of \emph{basic blocks of size $n$}, by
\[
\mathfrak{B}_n := \left\{\triangle_{n,1}, \, {\triangle_{n,1}}^t, \, \triangle_{n,3}, \, {\triangle_{n,3}}^t, \, {\triangle_{n,3}}^\circ, \,
( {\triangle_{n,3}}^t)^\circ  \right\}\, ,
\]
and $\mathfrak{B}$ to be the union of all $\mathfrak{B}_n$.
We call $\mathfrak{B}$ the set of {\it basic blocks}.
Note that if $n \leq 3$, then $\mathfrak{B}_n = \{ \triangle_{n,1}, \, {\triangle_{n,1}}^t \}$,
and all ribbons in $\mathfrak{B}$ are $p$-positive.

\begin{figure}[h]
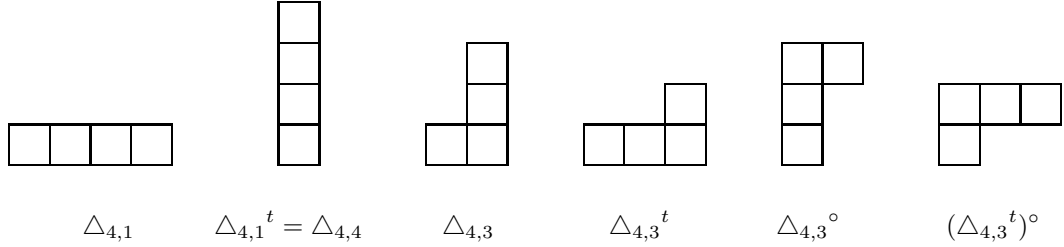

\centering
\small{
\begin{ytableau}
\none&\none&\none&\none\\
\none&\none&\none&\none\\
\none&\none&\none&\none\\
~&~&~&~
\end{ytableau}
\qquad \,\, \, \,
\begin{ytableau}
~\\
~\\
~\\
~
\end{ytableau}
\qquad \,\, \, \,
\begin{ytableau}
\none&\none\\
\none&~\\
\none&~\\
~&~
\end{ytableau}
\qquad \,
\begin{ytableau}
\none&\none&\none\\
\none&\none&\none\\
\none&\none&~\\
~&~&~
\end{ytableau}
\qquad \,
\begin{ytableau}
\none&\none\\
~&~\\
~\\
~
\end{ytableau}
\qquad \,
\begin{ytableau}
\none&\none&\none\\
\none&\none&\none\\
~&~&~\\
~
\end{ytableau}\\[2ex]
}

\[
\qquad\triangle_{4,1} \qquad \quad  {\triangle_{4,1}}^t=\triangle_{4,4} \qquad\quad   \triangle_{4,3} \qquad\qquad\,\, {\triangle_{4,3}}^t \qquad\qquad {\triangle_{4,3}}^\circ \qquad \qquad ( {\triangle_{4,3}}^t)^\circ
\]
    \caption{The elements of the set $\mathfrak{B}_4$ of basic blocks of size $4$.}  \label{fig:Basicblock}
\end{figure}

{We state the main theorem of this section, in which an infinite family of $p$-positive ribbons is constructed from our basic blocks.}

\begin{thm}\label{sufficient}
Let $\mathfrak{B}$ be the set of basic blocks.
If a ribbon $D$ is of the form
{\[
D = \alpha^{(k)} \bullet \cdots \bullet \alpha^{(2)} \bullet \alpha^{(1)} \bullet B
\]}
for some $B \in \mathfrak{B}$ and {a positive integer $k$ with compositions $\alpha^{(i)}\in \{ (2),  (1,1)\}$} for each $1 \leq i \leq k$,
then $\rib_D$ is $p$-positive.
\end{thm}
\begin{proof}
We first observe that
$(2)\bullet B= B \odot B^t$ and $(1,1)\bullet B= B \cdot B^t$.
Hence,
if we let {$D_i=\alpha^{(i)}\bullet\cdots\bullet \alpha^{(1)}\bullet B$} for $1 \leq i \leq k$,
then we can easily see that $\mathfrak r_{D_{i+1}}=\frac{1}{2} \mathfrak r_{D_i}^2$ for $ 0 \leq i \leq k-1$
by Proposition~\ref{prop:ribbon}\,\eqref{eqn:ribbonsq}.
It follows that $\mathfrak r_{D}=\frac{1}{2^k} (\mathfrak r_B)^{2^k}$,
and therefore $\mathfrak r_{D}$ is $p$-positive.
\end{proof}

We remark that we do not include $\triangle_{n,\frac{n}{2}}$, ${\triangle_{n,\frac{n}{2}}}^t(=\triangle_{n,\frac{n}{2}+1})$, ${\triangle_{n,\frac{n}{2}}}^{\circ}$, and
${({\triangle_{n,\frac{n}{2}}}^t)}^{\circ}$ in $\mathfrak{B}$
since all of them can be obtained from basic blocks by taking a composition of transposition as in Theorem~\ref{sufficient}.
For instance,
$\triangle_{n,\frac{n}{2}} = (2) \bullet {\triangle_{\frac{n}{2},1}}^t$.

\subsection{A conjecture on the classification of $p$-positive ribbons}

Our conjecture is that the converse of Theorem~\ref{sufficient} is also true. 
That is, for connected $\widetilde{\la/\mu}$, $Q_{\la/\mu}$ is $p$-positive if and only if $\widetilde{\la/\mu}$ is a ribbon obtained by applying finite number of  compositions of transposition by either $(1, 1)$ or $(2)$ to a basic block.
Remember that $\mathfrak{B}$ be the set of basic blocks.

\begin{conj}\label{necessary}
{For a connected $D$,} if $\rib_D$ is $p$-positive 
then either $D \in \mathfrak{B}$ or
\[
D = \alpha^{(k)} \bullet \cdots \bullet \alpha^{(2)} \bullet \alpha^{(1)} \bullet B
\]
for some $B \in \mathfrak{B}$ and {a positive integer $k$ with compositions $\alpha^{(i)}\in \{ (2),  (1,1)\}$} for $1 \leq i \leq k$.
\end{conj}

We verified {the conjecture for }all cases up to $n = 12$ via sage computation, and in what follows
we list a few evidences conversing to support Conjecture~\ref{necessary}.
To begin with, we recall that
for each ribbon $D$
there are four kinds of equivalent ribbons in the sense that
\[
\rib_D = \rib_{D^t} = \rib_{D^\circ} = \rib_{(D^t)^\circ} \,.
\]
We also recall that a ribbon $D$ corresponds {to} the composition $\alpha(D) = ( \alpha_1(D),  \alpha_2(D), \cdots, \alpha_\ell(D) )$ whose $i$th part is equal to the number of boxes on the $i$th row of $D$.
Note that if $\alpha_1(D) >1$ then $\alpha_1((D^t)^\circ) = 1$.
Since our goal is to determine {the} $p$-positivity of $\rib_D$,
it is enough to consider the ribbons with $\alpha_1(D) = 1$.
Hence, from now on,
we always assume that a ribbon has a only one box on the first row.
We say that a box $(i,j)$ {in $D$} is a {\it corner}  of $D$ if the box $(i-1, j)$ is contained in $D$ but $(i+1, j)$ is not in $D$. Let $c(D)$ be the number of all corners in $D$.

We summarize some identities that the coefficients in the $p$-expansion of ribbon Schur functions satisfy.

\begin{lem}\label{lem:corners}
Let $D$ be a connected ribbon of size $n$ with length $\ell$, and let $\rib_D = \sum_{\la \in OP(n)} c_\la \, p_\la$.
Then we have {the} followings.
\begin{enumerate}
\item[(a)] ${\sum_{\lambda \in OP(n)} c_\lambda = 2}$.
\item[(b)] {If $n$ is odd}, then $c_{(n)} = (-1)^{\ell+1} \,  2/n$.
\item[(c)] If $\alpha_1(D)=1$ and $\alpha_\ell (D) > 1$, then
 $\sum_{\lambda \in OP(n)}m_1(\lambda)c_{\lambda} = 8 c(D).$
\item[(d)] If $\alpha_1(D)=\alpha_\ell (D) = 1$, then
 $\sum_{\lambda \in OP(n)}m_1(\lambda)c_{\lambda}= 8c(D)- 4$ for $n
 \geq2$.
\end{enumerate}
\end{lem}

\begin{proof}
(a)
Since $D$ is a ribbon,
there are exactly two kinds of marked shifted tableaux of shape $D$ and weight $(n)$.
Thus, we have 
\[
\<\ \rib_{D},~ s_{(n)} \ra =
\left\<\sum_{\lambda \in OP(n)}c_{\lambda}p_{\lambda}, ~\sum_{\lambda \vdash n}z_{\lambda}^{-1} p_{\lambda} \right\ra
= \sum_{\lambda \in OP(n)}c_{\lambda} = 2 \, .
\]

(b) This is immediate from  Proposition~\ref{prop:ribbon}\,\eqref{eqn:ribbon} and  Proposition~\ref{prop:SchurQ}\,\eqref{eqn:q_to_p}.

(c)
Let $T$ be a marked shifted tableau of shape $D$ and weight $(n-1,1)$. We use $T(i, j)$ for the content filled in the box $(i, j)$ of $T$.
First, notice that if $T(i, j)$ is $2$ or $2'$,
then the box $(i,j)$ should be a corner of $D$.
Hence,
we can choose a box to place either $2$ or $2'$ among $c(D)$ corners.
Second,
if { $|T(i,j)| = 2$ then $T(i-1, j)$ can be either $1$ or $1'$.}
Lastly,
since $\alpha_\ell(D) > 1$ the leftmost box on the $\ell$th row of $D$ is not a corner 
and this box must be filled with either $1$ or $1'$.
In all,
we conclude that the number of  marked shifted tableaux of shape $D$ and weight $(n-1,1)$ is $8c(D)$,
and thus
\[
8c(D) = \< \rib_D, s_{(n)} + s_{(n-1,1)}\ra
= \< \rib_D, s_{(n)} \ra  + \< \rib_D,  s_{(n-1,1)}\ra  \, .
\]
Finally, using the following Murnaghan-Nakayama rule for $s_{(n-1,1)}$;
$$s_{(n-1,1)}=\sum_{\lambda \vdash n} z_{\lambda}^{-1}(m_1(\lambda)-1)p_{\lambda} ,$$
we deduce that
\[
  8c(D)  = \sum_{\la \in OP(n)}c_\la + \sum_{\lambda \in OP(n)} (m_1(\lambda)-1)\, c_\la \, = \,  \sum_{\lambda \in OP(n)} m_1(\lambda)\, c_\la \, .
\]

(d) This can be proved in a similar in (b).
But, in this case,
the unique box, say $b$, on the $\ell$th row of $D$ is a corner. 

Thus $b$ can be filled with $2$ or $2'$, and the number of such marked shifted tableaux of shape $D$ and weight $(n-1, 1)$ is 4.
This implies that
\[
\< \rib_D, ~s_{(n)}+s_{(n-1,1)} \ra = 8(c(D)-1)+4 = 8c(D) - 4\,.
\]
\end{proof}

The following corollary shows that
$\rib_D$ is not $p$-positive if $c(D)$ is sufficiently large.

\begin{cor}\label{cor:many_corners}
For a connected ribbon $D$ of size $n$,
if $c(D) > 1/2 +n/4$ then $\rib_D$ is not $p$-positive.
\end{cor}

\begin{proof}
By Lemma~\ref{lem:corners}, we have
$$\sum_{\lambda \in OP(n)}(n-m_1(\lambda)) \, c_{\lambda}\leq 2n- 8c(D) + 4  ,$$
which turns out to be negative if $c(D) > 1/2+n/4$.
Since  $m_1(\lambda)$ is at most $n$, the number $n-m_1(\lambda)$ is always nonnegative. Therefore, $c_{\lambda}<0$ for some $\lambda$.
\end{proof}

\begin{rem}
Let $D$ be a ribbon of the form
$D = {\alpha^{(k)} \bullet \cdots \bullet \alpha^{(2)} \bullet \alpha^{(1)} \bullet B}$ as in Conjecture~\ref{necessary}. 
Since we assume that $\alpha_1(D) = 1$,
$B$ should be a basic block satisfying that $\alpha_{1}(B) =1$. 
Let $\ell$ be the length of $B$.
If $\alpha_\ell(B) \neq 1$, then
the number of corners in $\alpha^{(i)} \bullet \alpha^{(i-1)} \bullet \cdots \bullet \alpha^{(1)} \bullet B$ 
is twice the number of corners in  $\alpha^{(i-1)} \bullet \cdots \bullet \alpha^{(1)} \bullet B$ for $1 \leq i \leq k$.
It follows that $c(D) = 2^k$ from $c(B) = 1$.
On the other hand, if $\alpha_\ell(B) = 1$, that is, $B$ is a one-row diagram, 
then $\alpha^{(1)} \bullet B = \triangle_{2n, n}$ or $\triangle_{2n, n+1}$,
and thus $c(\alpha^{(1)} \bullet B) = c(B) =1$. 
Hence, $c(D) = 2^{k-1}$.
Notice that 
$c(D) \leq 2^{k-1}$ if $2 \leq |B| \leq 3$ as $\alpha_\ell(B) = 1$
and $c(D) \leq 2^k$ otherwise.
In both cases, 
we can derive that $c(D) \leq |D|/4 < 1/2 + |D|/4$.
\end{rem}

According to Conjecture~\ref{necessary}, if the number of boxes $n=|D|$ of the ribbon is \emph{odd} then $\rib_D$ is $p$-positive only when $D$ is a basic block.
We prove that $\rib_D$ is not $p$-positive if $D\not\in \mathfrak{B}$ satisfies certain conditions, when $|D|$ is odd. 

When the length of the ribbon is even, non $p$-positivity of $\rib_D$ immediately follows from Lemma~\ref{lem:corners}.

\begin{thm}
Let $D\not\in \mathfrak{B}$ be a ribbon with odd number of boxes.
If the length of $D$ is even, then $\rib_D$ is not $p$-positive.
\end{thm}

For the ribbons of odd length, more work has to be done. We set up some necessary notions first.
For a given ribbon $D$ of length $\ell$,
we define the {\it head length} of $D$ by the smallest index $i$
such that the $i$th row of $D$ contains at least two boxes. We define the \emph{tail length} of a ribbon according to the number of boxes in the last row:  
If $\alpha_{\ell}(D)=1$, then we define the \emph{tail length} of $D$ by the smallest index $j$
such that $(\ell - j +1)$st row of $D$ contains at least two boxes.
Otherwise
we define the \emph{tail length} of $D$ by $\alpha_\ell(D)$.

\begin{example}
If $D_1 = (1,3,1,1,1)$, then the head length and the tail length of $D_1$ are 2 and 4, respectively. 
On the other hand, if $D_2 = (1,3,1,1,2)$, 
then the head length and the tail length of $D_2$ are 2 and 2, respectively. \\
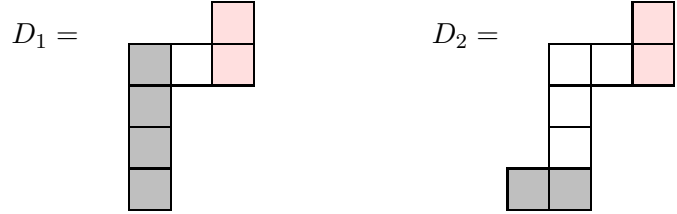
\begin{center}
$D_1=
\small{
\begin{ytableau}
\none&\none&\none& *(pink!50)\\
\none& *(gray!50) &~& *(pink!50)\\
\none&*(gray!50)\\
\none&*(gray!50)\\
\none&*(gray!50)
\end{ytableau}
}$\qquad \qquad \qquad
$D_2=
\small{
\begin{ytableau}
\none&\none&\none&  *(pink!50)\\
\none&~&~&  *(pink!50)\\
\none& ~\\
\none& ~\\
*(gray!50)&*(gray!50)
\end{ytableau}
}$
\end{center}
\end{example}

\begin{thm}
Let $D=(\alpha_1(D),\alpha_2(D), \dots, \alpha_{\ell}(D))\not\in \mathfrak{B}$ be a ribbon with  an odd number of boxes such that $\alpha_1 (D) =1$.
If the length of $D$ is odd, then  $\rib_D$ is not $p$-positive in the following cases, where $k$ and $m$ are the head length and the tail length of $D$, respectively.
\begin{enumerate}
    \item[(a)]   $\alpha_\ell (D) >1$ and $k$ is even.
    \item[(b)]   $\alpha_\ell (D) >1$, $k$ is odd and $m$ is even.
    \item[(c)]   $\alpha_\ell (D) =1$ and $k=m$.
\end{enumerate}
\end{thm}
\begin{proof}
We let $n$ be the number of boxes of $D$, that is assumed to be odd. 
Without loss of generality, we may assume that $m \geq k$. Let $\rib_D = \sum_{\la \in OP(n)} c_\la p_\la$. We use  Proposition~\ref{prop:SchurQ}\,\eqref{eqn:q_to_p}, Proposition~\ref{prop:ribbon} \eqref{eqn:ribbon} and Lemma~\ref{lem:simple} in the proofs that follow.

(a) To prove the assertion,
we show that the coefficient $c_\la$ is negative, where $\la = \la(n-k, k-1, 1)$. Note that $m=\alpha_\ell(D)>1$ in this case.

We first suppose that $k = 2$ so that $\la=(n-2,1,1)$. Then we have
\[
z_{\la}c_{\la}=
\begin{cases}
\<p_{\la},\,q_n- q_{n-1}q_1 - q_{n-2}q_2\ra =-16  & \text{if}\ m=2,\\[2ex]
\<p_{\la},\,q_n- q_{n-1}q_1 \ra =-8 & \text{if}\ m>3\,.
\end{cases}
\]

Now suppose that $k$ is even such that $k > 2$. Note that $n-k$, $k-1$ and $1$ are mutually distinct due to the assumptions we make.  Then
\[
z_{\la}c_{\la}=
\begin{cases}
\<p_{\la},\,q_n- q_{n-1}q_1 - q_{n-k+1} q_{k-1} - q_{n-k} q_{k}\ra =-16  & \text{if}\ m=k,\\[2ex]
\<p_{\la},\,q_n- q_{n-1}q_1 - q_{n-k+1} q_{k-1} \ra =-8 & \text{if} \ m>k \,.
\end{cases}
\]

(b) Now we consider odd $k$ and even $m=\alpha_\ell(D)>1$ and we trace the coefficient $c_\la$ with $\la = \la(n-m, m-1,1)$. Note that $m-1\neq 1$ Since $\alpha(D)=1$ implies $k> 1$, while it is possible that $n-m=m-1$.
\begin{itemize}
  \item
  If there is an index $i$ such that $\alpha_1(D) + \cdots + \alpha_{i-1}(D) = m-1$ and $\alpha_i(D) = 1$, then 
  \[
  z_{\la}c_{\la}=\<p_{\la},\,q_n- q_{n-1}q_1 - q_{n-m+1}q_{m-1} - 2 q_{n-m} q_m + 2q_{n-m} q_{m-1} q_1 \ra= \begin{cases}
    -16 & \text{if} \ n-m = m-1, \\[2ex]
    -8 & \text{otherwise}\,.
  \end{cases}\,.
  \]
  \item
   If there is an index $i$ such that $\alpha_1(D) + \cdots + \alpha_{i-1}(D) = m-1$ and $\alpha_i(D) > 1$, then 
  \[
  z_{\la}c_{\la}=\<p_{\la},\,q_n- q_{n-1} q_1 - q_{n-m+1} q_{m-1} - q_{n-m}q_m \ra= 
  \begin{cases}
    -32 & \text{if} \ n-m = m-1, \\[2ex]
    -16 & \text{otherwise}\,.
  \end{cases}
  \]
  \item
  If there is an index $i$ such that $\alpha_1(D) + \cdots + \alpha_{i}(D)= m$ 
  and $\alpha_i(D)>1$, then
  \[
  z_{\la}c_{\la}=\<p_{\la},\,q_n- q_{n-1}q_1 - 2q_{n-m} q_m + q_{n-m}q_{m-1} q_1\ra= \begin{cases}
    -16 & \text{if} \ n-m = m-1, \\[2ex]
    -8 & \text{otherwise}\,.
  \end{cases}\,.
  \]
  \item
  If there is no index $i$ such that $\alpha_1(D) + \cdots + \alpha_{i}(D) = m-1$ or $m$, then
  \[
  z_{\la}c_{\la}=\<p_{\la},\,q_n- q_{n-1} q_1- q_{n-m} q_m\ra=
  \begin{cases}
    -16 & \text{if} \ n-m = m-1, \\[2ex]
    -8 & \text{otherwise}\,.
  \end{cases}
  \]
 \end{itemize}

(c)
This can be proven in a similar way as in (b)
by tracing the coefficient $c_\la$,
where $\la = \la(n-m, m-1, 1)$ if $k = m$ is even
and $\la = \la(n-m-1, m, 1)$ otherwise.
\end{proof}

\vspace{3mm}

\section{Concluding Remarks}\label{sec:remarks} 
Conjecture~\ref{necessary} is about the $p$-positive \emph{connected} ribbons. For a ribbon $D$ having $r$ connected components, say $D=D_1\cup \cdots \cup D_r$, the ribbon Schur $Q$-function of $D$ is decomposed as $\rib_D=\prod_{i=1}^r \rib_{D_i}$. Since a product of $p$-positive symmetric functions is $p$-positive, we know that if $D$ is a ribbon having $p$-positive ribbons as its connected components then $\rib_D$ is $p$-positive. However, the converse is not true in general; that is the $p$-positivity of $\prod_{i=1}^r \rib_{D_i}$ does not guarantee the $p$-positivity of each $\rib_{D_i}$, and the classification of all $p$-positive ribbons needs more sophisticated work. We, nevertheless believe and make a conjecture, which we checked with Sage for some specific cases.
\begin{conj} Let $D=D_1\cup \cdots \cup D_r$ be a ribbon with connected components $D_i$, $i=1, \dots, r$. Then, 
$\rib_D$ is $p$-positive if and only if $\rib_{D_i}$ is $p$-positive for all $i$.
\end{conj}

\vspace{3mm}
\section*{Acknowledgements}\label{sec:acknow} 
The  authors  would  like  to  thank  Stephanie van Willigenburg for  helpful  conversations, Sogang University and Korea Institute for Advanced Study where some of the research took place.  

\vspace{3mm}
\bibliographystyle{plain}  
\bibliography{mybib} 

\begin{thebibliography}{1}

\bibitem{BW}
F.~Barekat and S.~van Willigenburg.
\newblock Composition of transpositions and equality of ribbon {S}chur
  {$Q$}-functions.
\newblock {\em Electron. J. Combin.}, 16(1):Research Paper 110, 28, 2009.

\bibitem{CW1}
S.~Cho and S.~van Willigenburg.
\newblock Chromatic bases for symmetric functions.
\newblock {\em Electron. J. Combin.}, 23(1):Paper 1.15, 7, 2016.

\bibitem{CW2}
S.~Cho and S.~van Willigenburg.
\newblock Chromatic classical symmetric functions.
\newblock {\em J. Comb.}, 9(2):401--409, 2018.

\bibitem{Mac}
I.~G. Macdonald.
\newblock {\em Symmetric functions and {H}all polynomials}.
\newblock Oxford Mathematical Monographs. The Clarendon Press, Oxford
  University Press, New York, second edition, 1995.
\newblock With contributions by A. Zelevinsky, Oxford Science Publications.

\bibitem{S1}
R.~P. Stanley.
\newblock A symmetric function generalization of the chromatic polynomial of a
  graph.
\newblock {\em Adv. Math.}, 111(1):166--194, 1995.

\bibitem{S2}
R.~P. Stanley.
\newblock {\em Enumerative combinatorics. {V}ol. 2}, volume~62 of {\em
  Cambridge Studies in Advanced Mathematics}.
\newblock Cambridge University Press, Cambridge, 1999.
\newblock With a foreword by Gian-Carlo Rota and appendix 1 by Sergey Fomin.

\end{thebibliography}
\end{document}